\documentclass[12pt]{article}%
\usepackage{amssymb}
\usepackage{amsfonts}
\usepackage{amsmath}%
\usepackage{graphicx}

\numberwithin{equation}{section}

\newtheorem{theorem}{Theorem}[section]

\newtheorem{lemma}[theorem]{Lemma}
\newtheorem{proposition}[theorem]{Proposition}
\newtheorem{remark}[theorem]{Remark}
\newenvironment{proof}[1][Proof]{\noindent\textbf{#1.} }{\ \rule{0.5em}{0.5em}}

\newcommand{\Lip}{\operatorname{Lip}}

\newcommand{\R}{{\mathbb R}}

\title{Weak Convergence to Stable L\'evy Processes for Nonuniformly Hyperbolic
Dynamical Systems}
\author{
Ian Melbourne\thanks{Mathematics Institute, University of Warwick, Coventry, CV4 7AL, UK  i.melbourne@warwick.ac.uk}
\and
Roland Zweim\"uller\thanks{Faculty of Mathematics, University of Vienna,
Vienna,   Austria}
}

\date{11 January 2013; final version 18 September 2013}

\begin{document}

\maketitle

\begin{abstract}
We consider weak invariance principles (functional limit theorems) in the domain of a stable law.   A general result is obtained on lifting such limit laws from an induced dynamical system to the original system.  An important class of examples covered by our result are Pomeau-Manneville intermittency maps, where convergence for the induced system is in the standard Skorohod $\mathcal{J}_1$ topology.  For the full system, convergence in the $\mathcal{J}_1$ topology fails, but we prove convergence in the $\mathcal{M}_1$ topology.
\end{abstract}

\textbf{Keywords:} Nonuniformly hyperbolic systems, functional limit theorems,
L\'evy processes, induced dynamical systems

\textbf{MSC codes:} 37D25;  28D05, 37A50, 60F17

\section{Introduction}

For large classes of dynamical systems with good mixing properties, it is
possible to obtain strong statistical limit laws such as the central
limit theorems and its refinements including the
almost sure invariance principle (ASIP)~\cite{HofbauerKeller82,DenkerPhilipp84,
ConzeBorgne01,FMT03,MT02,MN05,MN09,BM08,Gouezel10}.
An immediate consequence of the ASIP is the weak invariance principle (WIP) which is the focus of this paper.

Thus the standard WIP (weak convergence to Brownian motion) holds
for general Axiom~A diffeomorphisms and flows, and also for nonuniformly
hyperbolic maps and flows modelled by Young towers~\cite{Young98,Young99} with
square integrable return time function (including
H\'enon-like attractors~\cite{BenedicksYoung00}, finite horizon Lorentz gases~\cite{MN05}, and the Lorenz attractor~\cite{HM07}).

Recently, there has been interest in statistical limit laws for dynamical
systems with weaker mixing properties such as those modelled by a Young tower
where the return time function is not square integrable.   In the borderline case where the return time lies in $L^p$ for all $p<2$,
it is often possible to prove a central limit theorem with nonstandard norming
(nonstandard domain of attraction of the normal distribution).
This includes important examples such as the infinite horizon Lorentz
gas~\cite{SzaszVarju07},
the Bunimovich stadium~\cite{BalintGouezel06} and billiards with cusps~\cite{BalintChernovDolgopyat11}.   In such cases, it is also possible to obtain the corresponding WIP
(see for example~\cite{BalintChernovDolgopyat11,DedeckerMerlevede}).

For Young towers with non-square-integrable return time function,
the central limit theorem generally fails.   Gou\"ezel~\cite{Gouezel04} (see also Zweim\"uller~\cite{Zweimuller03}) obtained definitive results on convergence in distribution to stable laws.   The only available results on the corresponding
WIP are due to Tyran-Kami\'nska~\cite{Tyran-Kaminska10} who gives
necessary and sufficient conditions for weak convergence to the
appropriate stable L\'evy process in the standard
Skorohod $\mathcal{J}_1$ topology~\cite{Skorohod56}.   However
in the situations we are interested in, the $\mathcal{J}_1$
topology is too strong and the results in~\cite{Tyran-Kaminska10}
prove that weak convergence fails in this topology.

In this paper, we repair the situation by working with the
$\mathcal{M}_1$ topology (also introduced by Skorohod~\cite{Skorohod56}).
In particular, we give general conditions for
systems modelled by a Young tower, whereby
convergence in distribution to a stable law can be improved
to weak convergence in the $\mathcal{M}_1$ topology to the corresponding L\'evy process.

The proof is by inducing (see~\cite{Ratner73,MT04,Gouezel07} for proofs by inducing of convergence in distribution).    Young towers by definition have a good inducing
system, namely a Gibbs-Markov map (a Markov map with bounded distortion and big images~\cite{AaronsonDenker01}).   The results of Tyran-Kami\'nska~\cite{Tyran-Kaminska10} often apply positively for such induced maps (see for example the proof of Theorem~\ref{thm-const} below) and yield weak convergence in the $\mathcal{J}_1$ topology, and hence the $\mathcal{M}_1$
topology, for the induced system.    
The main theoretical result of
the present paper discusses how $\mathcal{M}_1$ convergence in an 
induced system lifts to the original system (even when convergence 
in the $\mathcal{J}_1$ topology does not lift).

As a special case, we recover
the aforementioned results~\cite{BalintChernovDolgopyat11,DedeckerMerlevede} on the WIP in the nonstandard domain of
attraction of the normal distribution.

In the remainder of the introduction, we describe how our results
apply to Pomeau-Manneville intermittency maps~\cite{PomeauManneville80}.
In particular, we consider the family of maps $f:X\to X$, $X=[0,1]$,
studied by~\cite{LiveraniSaussolVaienti99}, given by
\begin{align} \label{eq-LSV}
f(x)=\begin{cases}   x(1+2^\gamma x^\gamma), & x\in[0,\frac12] \\
2x-1, & x\in(\frac12,1]
\end{cases}
\end{align}
For $\gamma\in[0,1)$, there is a unique absolutely continuous ergodic invariant probability measure $\mu$.
Suppose that $\phi:X\to\R$ is a H\"older observable with $\int_X \phi\,d\mu=0$.
Let $\phi_n=\sum_{j=0}^{n-1}\phi\circ f^j$.
For the map in~\eqref{eq-LSV}, our main result implies the following:

\begin{theorem} \label{thm-LSV}
Let $f:[0,1]\to[0,1]$ be the map~\eqref{eq-LSV} with $\gamma\in(\frac12,1)$
and set $\alpha=1/\gamma$.    Let $\phi:[0,1]\to\R$ be a mean zero H\"older observable and suppose that $\phi(0)\neq0$.
Define $W_n(t)=n^{-1/\alpha}\phi_{\lfloor nt\rfloor}$.
Then $W_n$ converges weakly in the Skorohod $\mathcal{M}_1$ topology to an $\alpha$-stable L\'evy
process.  (The specific L\'evy process is described below.)
\end{theorem}

\begin{remark} \label{rmk-JM}
The $\mathcal{J}_1$ and $\mathcal{M}_1$ topologies are reviewed in Section~\ref{sec-JM}.   Roughly speaking, the difference is that the $\mathcal{M}_1$
topology allows numerous small jumps for $W_n$ to accumulate into a large
jump for $W$, whereas the $\mathcal{J}_1$ topology would require a large jump
for $W$ to be approximated by a single large jump for $W_n$.
Since the jumps in $W_n$ are bounded by $n^{-1/\alpha}|\phi|_\infty$, it is evident that in Theorem~\ref{thm-LSV} convergence cannot hold
in the $\mathcal{J}_1$ topology.

Situations in the probability theory literature where convergence holds in
the $\mathcal{M}_1$ topology but not the $\mathcal{J}_1$ topology include
\cite{AvramTaqqu92, BenArousCerny07}.
\end{remark}

Theorem~\ref{thm-LSV} completes the study of weak convergence for the
intermittency map~\eqref{eq-LSV} with $\gamma\in[0,1)$ and typical
H\"older observables.
We recall the previous results
in this direction.
If $\gamma\in[0,\frac12)$ then it is well-known that $\phi$ satisfies
a central limit theorem, so
$n^{-\frac12}\phi_n$ converges in distribution to a normal distribution with mean
zero and variance $\sigma^2$, where $\sigma^2$ is typically positive.
Moreover,~\cite{MN05} proved the ASIP.
An immediate consequence
is the WIP:
$W_n(t)=n^{-\frac12}\phi_{\lfloor nt\rfloor}$ converges weakly to Brownian motion.

If $\gamma=\frac12$ and $\phi(0)\neq0$, then Gou\"ezel~\cite{Gouezel04}
proved that $\phi$ is in the nonstandard domain of attraction of the
normal distribution: $(n\log n)^{-\frac12}\phi_n$ converges in distribution to a normal
distribution with mean zero and variance $\sigma^2>0$.
Dedecker \& Merlevede~\cite{DedeckerMerlevede} obtained the corresponding WIP in this situation (with $W_n(t)=(n\log n)^{-\frac12}\phi_{\lfloor nt\rfloor }$).

Finally, if $\gamma\in(\frac12,1)$ and $\phi(0)\neq0$, then Gou\"ezel~\cite{Gouezel04} proved that $n^{-1/\alpha}\phi_n$ converges in distribution to a one-sided stable
law $G$ with exponent $\alpha=\gamma^{-1}$.
The stable law in question has characteristic function
\[
E(e^{itG})=\exp\bigl\{-c|t|^{\alpha}(1-i\,{\rm sgn}(\phi(0)t)\tan(\alpha\pi/2)\bigr\},
\]
where $c=\frac14 h(\frac12)(\alpha|\phi(0)|)^\alpha\Gamma(1-\alpha)\cos(\alpha\pi/2)$ and $h=\frac{d\mu}{dx}$ is the invariant density.
Let $\{W(t);\,t\ge0\}$ denote the corresponding $\alpha$-stable L\'evy process
(so $\{W(t)\}$ has independent and stationary increments with cadlag sample
paths and $W(t)=_d t^{1/\alpha}G$).
Tyran-Kami\'nska~\cite{Tyran-Kaminska10} verified that
$W_n(t)=n^{-1/\alpha}\phi_{\lfloor nt\rfloor}$
does not converge weakly to $W$ in the $\mathcal{J}_1$ topology.
In contrast, Theorem~\ref{thm-LSV} shows that
$W_n$ converges weakly to $W$ in the $\mathcal{M}_1$ topology.

The remainder of this paper is organised as follows.
In Section~\ref{sec-therewasanoldmanfromJapan_whowroteversethatneverwouldscan_whenaskedwhyitwas_hesaiditsbecause_IliketogetasmanywordsintothelastlineasIpossiblycan} we state our main abstract result,
Theorem~\ref{T_InducingWIP}, on inducing the WIP.
In Section~\ref{sec-proof} we prove Theorem~\ref{T_InducingWIP}.
In Section~\ref{sec-ex} we consider some examples which include Theorem~\ref{thm-LSV} as a special case.

\section{Inducing a weak invariance principle}
\label{sec-therewasanoldmanfromJapan_whowroteversethatneverwouldscan_whenaskedwhyitwas_hesaiditsbecause_IliketogetasmanywordsintothelastlineasIpossiblycan}

In this section, we formulate our main abstract result Theorem~\ref{T_InducingWIP}.
The result is stated in Subsection~\ref{Sec_TheInducingThm} after some preliminaries
in Subsection~\ref{sec-JM}.

\subsection{Preliminaries}
\label{sec-JM}

\textbf{Distributional convergence.} To fix notations, let $(X,P)$ be a
probability space and $(R_{n})_{n\geq1}$ a sequence of
Borel measurable maps $R_{n}:X\rightarrow S$, where $(S,d)$ is a separable
metric space. Then distributional convergence of $(R_{n})_{n\geq1} $ w.r.t.\
$P$ to some random element $R$ of $S$ will be denoted by $R_{n}\overset
{P}{\Longrightarrow}R$. \emph{Strong distributional convergence}%
\textit{\ }$R_{n}\overset{\mathcal{L}(\mu)}{\Longrightarrow}R$ on a
measure space $(X,\mu)$ means that $R_{n}\overset{P}{\Longrightarrow}R$
for all probability measures $P\ll\mu$.%

\vspace{0.3cm}%

\noindent
\textbf{Skorohod spaces.} We briefly review the required background material
on the Skorohod $\mathcal{J}_{1}$ and $\mathcal{M}_{1}$ topologies~\cite{Skorohod56} on the
linear spaces $\mathcal{D}[0,T]$, $\mathcal{D}[T_{1},T_{2}]$, and
$\mathcal{D}[0,\infty)$ of real-valued \emph{cadlag} functions (right-continuous
$g(t^{+})=g(t)$ with left-hand limits $g(t^{-})$) on the respective interval,
referring to~\cite{Whitt} for proofs and further information. Both topologies are
Polish, with $\mathcal{J}_{1}$ stronger than $\mathcal{M}_{1}$.

It is customary to first deal with bounded time intervals. We thus fix some
$T>0$ and focus on $\mathcal{D}=\mathcal{D}[0,T]$. (Everything carries over to
$\mathcal{D}[T_{1},T_{2}]$ in an obvious fashion.) Throughout, $\left\Vert
.\right\Vert $ will denote the uniform norm. Two functions $g_{1},g_{2}%
\in\mathcal{D}$ are close in the $\mathcal{J}_{1}$\emph{-topology} if they are
uniformly close after a small distortion of the domain. Formally, let
$\Lambda$ be the set of increasing homeomorphisms $\lambda:[0,T]\rightarrow
\lbrack0,T]$, and let $\lambda_{id}\in\Lambda$ denote the identity. Then
$d_{\mathcal{J}_{1},T}(g_{1},g_{2})=\inf_{\lambda\in\Lambda}\left\{
\left\Vert g_{1}\circ\lambda-g_{2}\right\Vert \vee\left\Vert \lambda
-\lambda_{id}\right\Vert \right\}  $ defines a metric on $\mathcal{D}$ which
induces the $\mathcal{J}_{1}$-topology. While its restriction to
$\mathcal{C}=\mathcal{C}[0,T]$ coincides with the uniform topology,
discontinuous functions are $\mathcal{J}_{1}$-close to each other if they have
jumps of similar size at similar positions.

In contrast, the $\mathcal{M}_{1}$\emph{-topology} allows a function $g_{1}$
with a jump at $t$ to be approximated arbitrarily well by some continuous
$g_{2}$ (with large slope near $t$).\ For convenience, we let $[a,b]$ denote
the (possibly degenerate) closed interval with endpoints $a,b\in\mathbb{R}$,
irrespective of their order. Let $\Gamma(g):=\{(t,x)\in\lbrack0,T]\times
\mathbb{R}:x\in\lbrack g(t^{-}),g(t)]\}$ denote the completed graph of $g$,
and let $\Lambda^{\ast}(g)$ be the set of all its parametrizations, that is,
all continuous $G=(\lambda,\gamma):[0,T]\rightarrow\Gamma(g)$ such that
$t^{\prime}<t$ implies either $\lambda(t^{\prime})<\lambda(t)$ or
$\lambda(t^{\prime})=\lambda(t)$ plus $\left\vert \gamma(t)-g(\lambda
(t))\right\vert \leq \left\vert \gamma(t^{\prime})-g(\lambda(t))\right\vert $.
Then $d_{\mathcal{M}_{1},T}(g_{1},g_{2})=\inf_{G_{i}=(\lambda_{i},\gamma
_{i})\in\Lambda^{\ast}(g_{i})}\left\{  \left\Vert \lambda_{1}-\lambda
_{2}\right\Vert \vee\left\Vert \gamma_{1}-\gamma_{2}\right\Vert \right\}  $
gives a metric inducing $\mathcal{M}_{1}$.

On the space $\mathcal{D}[0,\infty)$ the $\mathcal{\tau}$-topology, $\tau
\in\{\mathcal{J}_{1},\mathcal{M}_{1}\}$, is defined by the metric
$d_{\tau,\infty}(g_{1},g_{2}):=\int_{0}^{\infty}e^{-t}(1\wedge d_{\tau
,t}(g_{1},g_{2}))\,dt$. Convergence $g_{n}\rightarrow g$ in $(\mathcal{D}%
[0,\infty),\tau)$ means that $d_{\tau,T}(g_{n},g)\rightarrow0$ for every
continuity point $T$ of $g$.

For either topology, the corresponding Borel $\sigma$-field $\mathcal{B}%
_{\mathcal{D},\mathcal{\tau}}$ on $\mathcal{D}$, generated by the $\tau$-open
sets, coincides with the usual $\sigma$-field $\mathcal{B}_{\mathcal{D}}$
generated by the canonical projections $\pi_{t}(g):=g(t)$. Therefore, any
family $W=(W_{t})_{t\in\lbrack0,T]}$ or $(W_{t})_{t\in\lbrack0,\infty)}$ of
real random variables $W_{t}$ such that each path $t\mapsto W_{t}$ is cadlag,
can be regarded as a random element of $\mathcal{D}$, equipped with
$\tau=\mathcal{J}_{1}$ or $\mathcal{M}_{1}$.

\subsection{Statement of the main result}
\label{Sec_TheInducingThm}

Recall that for any ergodic measure preserving transformation
(\emph{m.p.t.}) $f$ on a probability
space $(X,\mu)$, and any $Y\subset X$ with $\mu(Y)>0$, the
\emph{return time} function $r:Y\rightarrow\mathbb{N}\cup\{\infty\}$ given by
$r(y):=\inf\{k\geq1:f^k(y)\in Y\}$ is integrable with mean $\int_{Y}%
r\,d\mu_{Y}=\mu(Y)^{-1}$ (Kac' formula), where $\mu_{Y}(A):=\mu(Y\cap
A)/\mu(Y)$.
Moreover, the \emph{first return map} or \emph{induced map} $F:=f^{r}:Y\rightarrow Y$ is
an ergodic m.p.t.\ on the probability space $(Y,\mu_{Y})$. This
is widely used as a tool in the study of complicated systems, where $Y$ is
chosen in such a way that $F$ is more convenient than $f$. In particular,
given an \emph{observable} (i.e.\ a measurable function) $\phi:X\rightarrow
\mathbb{R}$, it may be easier to first consider its \emph{induced version}
$\Phi:Y\rightarrow\mathbb{R}$ on $Y$, given by $\Phi:=\sum_{\ell=0}^{r-1}%
\phi\circ f^{\ell}$. By standard arguments, if $\phi\in L^1(X,\mu)$ then
$\Phi\in L^1(Y,\mu_Y)$ and $\int_Y \Phi \, d\mu_Y  = \mu(Y)^{-1} \int_X \phi \, d\mu$.
In this setup, we will denote the corresponding ergodic
sums by $\phi_k:=\sum_{\ell=0}^{k-1}\phi\circ f^{\ell}$\ and $\Phi_{n}%
:=\sum_{j=0}^{n-1}\Phi\circ F^{j}$, respectively.

Our core result allows us to pass from a weak invariance principle for the
induced version to one for the original observable. Such a step requires some
a priori control of the behaviour of ergodic sums $\phi_k$ during an
excursion from $Y$. We shall express this in terms of the function $\Phi
^{\ast}:Y\rightarrow\lbrack0,\infty]$ given by
\begin{align*}
\Phi^{\ast}(y):=\Bigl(\underset{0\leq\ell^{\prime}\le\ell\leq r(y)}{\max}\left(
\phi_{\ell^{\prime}}(y)-\phi_{\ell}(y)\right)  \Bigr)\wedge\Bigl(\underset{0\leq
\ell^{\prime}\le\ell\leq r(y)}{\max}\left(  \phi_{\ell}(y)-\phi_{\ell^{\prime}%
}(y)\right) \Bigr)\text{.}%
\end{align*}
Note that $\Phi^{\ast}$ vanishes if and only if the ergodic sums $\phi_k$
grow monotonically (nonincreasing or nondecreasing) during each excursion.
Hence bounding $\Phi^{\ast}$ means limiting the growth of $\phi_{\ell}$ until the
first return to $Y$ in at least one direction.
The expression $\Phi^\ast$ can be understood also in terms of the maximal
and minimal processes $\phi^\uparrow_\ell,\phi^\downarrow_\ell$ defined during each excursion
$0\le\ell\le r(y)$ by
\[
\phi^\uparrow_\ell(y)=\max_{0\le \ell^\prime\le \ell}\phi_{\ell^\prime}(y),\quad
\phi^\downarrow_\ell(y)=\min_{0\le \ell^\prime\le \ell}\phi_{\ell^\prime}(y).
\]
\begin{proposition} \label{prop-predominant}
\mbox{}
\newline
\noindent (i) In the ``predominantly increasing'' case
$\Phi^{\ast}(y)=\underset{0\leq\ell^{\prime}\le\ell\leq r(y)}{\max}\left(
\phi_{\ell^{\prime}}(y)-\phi_{\ell}(y)\right)$, we have
$\Phi^{\ast}(y)=\underset{0\le\ell\le r(y)}{\max}\left(\phi^\uparrow_\ell(y)-\phi_\ell(y)\right)$.

\vspace{2ex}
\noindent
(ii) In the ``predominantly decreasing'' case
$\Phi^\ast(y)= \underset{0\leq
 \ell^{\prime}\le\ell\leq r(y)}{\max}\left(  \phi_{\ell}(y)-\phi_{\ell^{\prime}%
 }(y)\right)$, we have
$\Phi^{\ast}(y)=\underset{0\le\ell\le r(y)}{\max}\left(\phi_\ell(y)-\phi^\downarrow_\ell(y)\right)$.
\end{proposition}

\begin{proof}
This is immediate from the definition of $\phi^\uparrow_\ell$ and
$\phi^\downarrow_\ell$.
\end{proof}

\vspace{0.3cm}%

We use $\Phi^{\ast}$ to impose
a weak monotonicity condition for $\phi_{\ell}$ during excursions.

\begin{theorem}
[\textbf{Inducing a weak invariance principle}]\label{T_InducingWIP}Let $f$ be
an ergodic m.p.t.\ on the probability space $(X,\mu)$, and
let $Y\subset X$ be a subset of positive measure with return time $r$ and first
return map $F$. Suppose that the observable $\phi:X\rightarrow\mathbb{R}$ is
such that its induced version $\Phi$ satisfies a WIP on $(Y,\mu_{Y})$ in that
\begin{equation}
\left(  P_{n}(t)\right)  _{t\geq0}:=\left(  \frac{\Phi_{\left\lfloor
tn\right\rfloor }}{B(n)}\right)  _{t\geq0}\overset{\mathcal{L}(\mu_{Y}%
)}{\Longrightarrow}\left(  W(t)\right)  _{t\geq0}\text{\quad in }%
(\mathcal{D}[0,\infty),\mathcal{M}_{1})\text{,}\label{Eq_WIPinduced}%
\end{equation}
where $B$ is regularly varying of index $\gamma>0$, and $\left(  W(t)\right)
_{t\geq0}$ is a process with cadlag paths. Moreover, assume that
\begin{equation}
\frac{1}{B(n)}\left(  \max_{0\leq j\leq n}\Phi^{\ast}\circ F^{j}\right)
\overset{\mu_{Y}}{\Longrightarrow}0\text{.}\label{Eq_EssMonotonicity}%
\end{equation}
Then $\phi$ satisfies a WIP on $(X,\mu)$ in that
\begin{equation}
\left(  W_{n}(s)\right)  _{s\geq0}:=\left(  \frac{\phi_{\left\lfloor
sn\right\rfloor }}{B(n)}\right)  _{s\geq0}\overset{\mathcal{L}(\mu
)}{\Longrightarrow}\left(  W(s\mu(Y))\right)  _{s\geq0}\text{\ in
}(\mathcal{D}[0,\infty),\mathcal{M}_{1})\text{.}\label{Eq_InducedWIPinM1}%
\end{equation}

\end{theorem}

\begin{remark}[$\alpha$-stable processes]  If the process $W$ in~\eqref{Eq_WIPinduced} for the induced
system is an $\alpha$-stable L\'evy process,
then the limiting process in~\eqref{Eq_InducedWIPinM1} is $(\int_Y r\,d\mu_Y)^{-1/\alpha}W$.
\end{remark}

\begin{remark}
In general, the convergence from (\ref{Eq_InducedWIPinM1}) fails in
$(\mathcal{D}[0,\infty),\mathcal{J}_{1})$, even if (\ref{Eq_WIPinduced}) holds
in the $\mathcal{J}_{1}$-topology. That this is the case for the intermittent
maps~\eqref{eq-LSV} was pointed out in \cite[Example 2.1]{Tyran-Kaminska10}.
\end{remark}

\begin{remark}[Continuous sample paths]
If the process $W$ in~\eqref{Eq_WIPinduced} for the induced system has continuous sample paths, then the statement and proof of Theorem~\ref{T_InducingWIP} is greatly
simplified and the uniform topology (corresponding to the uniform norm $\|\;\|$)
can be used throughout.   In particular, the function $\Phi^{\ast}$ is replaced by
$\Phi^{\ast}(y)=\max_{0\le\ell<r(y)}|\phi_\ell(y)|$.   In the case of normal diffusion $B(n)=n^{\frac12}$, condition~\eqref{Eq_EssMonotonicity} is then satisfied if $\Phi^{\ast}\in L^2$.

A simplified proof based on the one presented here is written out in~\cite[Appendix]{GMsub}.
\end{remark}

\begin{remark}[Centering]
In the applications that we principally have in mind (including the maps~\eqref{eq-LSV}),
the observable $\phi:X\to\R$ is integrable, and hence so is its induced version $\Phi:Y\to\R$.
In particular, if $\phi$ has mean zero, then $\Phi$ has mean zero and we are
in a situation to apply Theorem~\ref{T_InducingWIP}.
From this, it follows easily that if condition~\eqref{Eq_WIPinduced} holds with
\[
P_n(t)=\frac{\Phi_{\lfloor tn\rfloor}-tn \int_Y\Phi\,d\mu_Y}{B(n)},
\]
and condition~\eqref{Eq_EssMonotonicity} holds with $\phi$ replaced throughout by
$\phi-\int_X\phi\,d\mu$ in the definition of $\Phi^{\ast}$,
then conclusion~\eqref{Eq_InducedWIPinM1} is valid with
\[
W_n(s)=\frac{\phi_{\lfloor sn\rfloor}-sn \int_X\phi\,d\mu}{B(n)}.
\]

With a little more effort it is also possible to handle more general centering sequences where
the process $(P_n)$ in condition~\eqref{Eq_WIPinduced} takes the form
\[
P_n(t)=\frac{\Phi_{\lfloor tn\rfloor}-tA(n)}{B(n)},
\]
for real sequences $A(n)$, $B(n)$ with $B(n)\to\infty$.
\end{remark}

The monotonicity condition (\ref{Eq_EssMonotonicity}) will be shown to hold,
for example, if we have sufficiently good pointwise control for single
excursions:

\begin{proposition}
[\textbf{Pointwise weak monotonicity}]\label{P_PointwiseMonotonicity}Let $f$
be an ergodic m.p.t.\ on the probability space $(X,\mu)$, and
let $Y\subset X$ be a subset of positive measure with return time $r$. Let $B$ be
regularly varying of index $\gamma>0$. Suppose that for the observable
$\phi:X\rightarrow\mathbb{R}$ there is some $\eta\in(0,\infty)$ such that for
a.e.\ $y\in Y$,
\begin{equation}
\Phi^{\ast}(y)\leq\eta\,B(r(y))\text{.}\label{Eq_PointwiseMonotonicityCtrl}%
\end{equation}
Then the weak monotonicity condition (\ref{Eq_EssMonotonicity}) holds.
\end{proposition}

The proofs of Theorem~\ref{T_InducingWIP} and Proposition~\ref{P_PointwiseMonotonicity} are given in Section~\ref{sec-proof}.

Assuming strong distributional convergence $\overset{\mathcal{L}(\mu_{Y}%
)}{\Longrightarrow}$ in (\ref{Eq_WIPinduced}), rather than $\overset{\mu_{Y}%
}{\Longrightarrow}$, is not a restriction, as an application of the following
result to the induced system $(Y,\mu_{Y},F)$ shows.

\begin{proposition}
[\textbf{Automatic strong distributional convergence}]%
\label{P_AutomaticStrongDistrCge}Let $f$ be an ergodic m.p.t.\ on a $\sigma
$-finite space $(X,\mu)$.
Let $\tau=\mathcal{J}_{1}$ or $\mathcal{M}_1$ and let
$A(n),B(n)$ be real sequences with $B(n)\to\infty$,
Assume that $\phi:X\rightarrow
\mathbb{R}$ is measurable, and that
there is some probability measure $P\ll\mu$
and some random element $R$ of $\mathcal{D}[0,\infty)$ such that
\begin{equation}
R_{n}:=\left(  \frac{\phi_{\left\lfloor tn\right\rfloor }-tA(n)}{B(n)}\right)
_{t\geq0}\overset{P}{\Longrightarrow}R\text{\quad in }(\mathcal{D}%
[0,\infty),\tau)\text{,}\label{Eq_abcdefg}%
\end{equation}
Then, $R_{n}\overset{\mathcal{L}(\mu)}{\Longrightarrow}R$\ in
$(\mathcal{D}[0,\infty),\tau)$.
\end{proposition}

\begin{proof}
This is based on ideas in~\cite{Eagleson76}.
According to Zweim\"uller~\cite[Theorem 1]{Zweimueller07}, it suffices to check
that $d_{\tau,\infty}(R_{n}\circ f,R_{n})\overset{\mu}{\longrightarrow}0$. The
proof of \cite[Corollary 3]{Zweimueller07} shows that
$B(n)\rightarrow\infty$ alone (that is, even without (\ref{Eq_abcdefg}))
implies $d_{\mathcal{J}_{1},\infty}(R_{n}\circ f,R_{n})\overset{\mu
}{\longrightarrow}0$. Since $d_{\mathcal{M}_{1},\infty}\leq d_{\mathcal{J}%
_{1},\infty}$ (see \cite[Theorem 12.3.2]{Whitt}), the case $\tau=\mathcal{M}_{1}$
then is a trivial consequence.
\end{proof}

\begin{remark}
There is a systematic typographical error in~\cite{Zweimueller07}
 in that the factor $t$ in the centering process
$tA(n)/B(n)$ is missing, but the arguments there work, without any change, for
the correct centering.
\end{remark}%

\section{Proof of Theorem \ref{T_InducingWIP}.\label{Sec_PfofThm1}}
\label{sec-proof}

In this section, we give the proof of Theorem~\ref{T_InducingWIP}
and also Proposition~\ref{P_PointwiseMonotonicity}.
Throughout, we assume the setting of Theorem~\ref{T_InducingWIP}.
In particular, we suppose that $f$ is
an ergodic m.p.t.\ on the probability space $(X,\mu)$, and that
$Y\subset X$ is a subset of positive measure with return time $r$ and first
return map $F$.

\subsection{Decomposing the processes.}
When $Y$ is chosen
appropriately, many features of $f$ are reflected in the behaviour of the
ergodic sums $r_{n}=\sum_{j=0}^{n-1}r\circ F^{j}$, i.e.\ the times at which
orbits return to $Y$. These are intimately related to the
\emph{occupation times} or \emph{lap numbers}
\[
N_k:=\sum_{\ell=1}^k1_{Y}\circ f^{\ell}
=\max\{n\geq0:r_{n}\leq k\}\le k,\enspace k\ge0.
\]
The visits to $Y$ counted by the $N_k$ separate
the consecutive \emph{excursions from} $Y$, that is, the intervals
$\{r_{j},\ldots,r_{j+1}-1\}$, $j\geq0$. Decomposing the $f$-orbit of $y$ into
these excursions, we can represent the ergodic sums of $\phi$ as
\begin{align*}
\phi_k=\Phi_{N_k}+R_k\quad\text{on }Y
\end{align*}
with remainder term $R_k=\sum_{\ell=r_{N_k}}^{k-1}\phi\circ f^{\ell}%
=\phi_{k-r_{N_k}}\circ F^{N_k}$ encoding the contribution of the
incomplete last excursion (if any). Next, decompose the rescaled processes
accordingly, writing
\begin{align*}
W_k(s)=U_k(s)+V_k(s)\text{,}
\end{align*}
with $U_k(s):=B(k)^{-1}\Phi_{N_{\left\lfloor sk\right\rfloor }}$, and
$V_k(s):=B(k)^{-1}R_{_{\left\lfloor sk\right\rfloor }}$. On the time scale
of $U_{n}$, the excursions correspond to the intervals $[t_{n,j},t_{n,j+1})$,
$j\geq0$, where $t_{n,j}:Y\rightarrow\lbrack0,\infty)$\ is given by
$t_{n,j}:=r_{j}/n$. Note that the interval containing a given point $t>0$ is
that with $j=N_{\left\lfloor tn\right\rfloor }$. Hence
\begin{equation}
t\in\lbrack t_{n,N_{\left\lfloor tn\right\rfloor }},t_{n,N_{\left\lfloor
tn\right\rfloor }+1})\text{\quad for }t>0\text{ and }n\geq1\text{.}%
\label{Eq_ReturnTimesOnScaleOfU}%
\end{equation}

\subsection{Some almost sure results}

In this subsection, we record some consequences of the ergodic theorem which we will use
below. But first an elementary observation, the proof of which we omit.

\begin{lemma}
\label{L_HistoryOfSequence}Let $(c_{n})_{n\geq0}$ be a sequence in
$\mathbb{R}$ such that $n^{-1}c_{n}\rightarrow c\in\mathbb{R}$. Define a
sequence of functions $C_{n}:[0,\infty)\rightarrow\mathbb{R}$ by letting
$C_{n}(t):=n^{-1}c_{\left\lfloor tn\right\rfloor }-tc$. Then, for any $T>0$,
$(C_{n})_{n\geq1}$ converges to $0$ uniformly on $[0,T]$.
\end{lemma}

For the occupation times of $Y$, we then obtain:

\begin{lemma}
[\textbf{Strong law of large numbers for occupation times}]\label{L_LapNumberEtc} The occupation times $N_k$ satisfy
\newline\textbf{a)}
$k^{-1}N_k\longrightarrow\mu(Y)\text{\quad a.e. on }X\text{\quad as
}k\rightarrow\infty$.
\newline\textbf{b)} Moreover, for any $T>0$,
\begin{align*}
\sup\nolimits_{t\in\lbrack0,T]}\left\vert k^{-1}N_{\left\lfloor
tk\right\rfloor }-t\mu(Y)\right\vert \longrightarrow0\text{\quad a.e. on
}X\text{\quad as }k\rightarrow\infty\text{.}%
\end{align*}

\end{lemma}

\begin{proof}
The first statement is immediate from the ergodic theorem. The second then
follows by the preceding lemma.
\end{proof}%

\begin{lemma}
\label{L_ReturnTimeControl}
For any $T>0$,
$\,\lim_{n\to\infty}n^{-1}\max\nolimits_{0\leq j\leq\left\lfloor Tn\right\rfloor +1}(r\circ
F^{j})=0\text{\enspace a.e.\ on }Y$.
\end{lemma}

\begin{proof}
Applying the ergodic theorem to $F$ and the integrable function $r$, we get
$n^{-1}\sum_{j=0}^{n-1}r\circ F^{j}\rightarrow\mu(Y)^{-1}$, and hence also
$n^{-1}(r\circ F^{n})\rightarrow0$ a.e. on $Y$. The result follows from Lemma
\ref{L_HistoryOfSequence}.
\end{proof}%

\vspace{0.3cm}%
\noindent
\textbf{Pointwise control of monotonicity behaviour.} We conclude this subsection
by establishing Proposition \ref{P_PointwiseMonotonicity}.

\vspace{0.3cm}%

\begin{proof}
[\textbf{Proof of Proposition \ref{P_PointwiseMonotonicity}}]
We may suppose without loss that the sequence $B(n)$ is nondecreasing.   Since this sequence is regularly varying,
$B(\widehat{\delta}n)/B(n)\rightarrow\widehat{\delta}^{\gamma}$ for
all $\widehat{\delta}>0$.
Hence for
$\delta>0$ fixed, there are $\widehat{\delta}>0$ and $\widehat{n}\geq1$ s.t.
$\eta\,B(h)/B(n)<\delta$\ whenever $n\geq\widehat{n}$ and $h\leq
\widehat{\delta}n$.

As a consequence of Lemma \ref{L_ReturnTimeControl}, there is some
$\widetilde{n}\geq1$ such that $Y_{n}:=\{n^{-1}\max\nolimits_{0\leq j\leq
n}(r\circ F^{j})<\widehat{\delta}\}$\ satisfies $\mu_{Y}(Y_{n}^{c}%
)<\varepsilon$ for $n\geq\widetilde{n}$. In view of
(\ref{Eq_PointwiseMonotonicityCtrl}) we then see (using monotonicity of $B$
again) that
\begin{align*}
\frac{1}{B(n)}\left(  \max_{0\leq j\leq n}\Phi^{\ast}\circ F^{j}\right)    &
\leq\frac{1}{B(n)}\left(  \max_{0\leq j\leq n}\eta\,B(r\circ F^{j})\right)
\\
& \leq\frac{\eta\,B(\max\nolimits_{0\leq j\leq n}(r\circ F^{j}))}%
{B(n)}\,<\delta\text{\quad on }Y_{n}\text{\quad for }n\geq\widehat{n}\text{,}%
\end{align*}
which proves (\ref{Eq_EssMonotonicity}).
\end{proof}%

\subsection{Convergence of $(U_{n})$.}
As a first step towards Theorem
\ref{T_InducingWIP}, we prove that switching from $\Phi_{\left\lfloor
tn\right\rfloor }$ to $\Phi_{N_{\left\lfloor sk\right\rfloor }}$ preserves
convergence in the Skorohod space.

\begin{lemma}
[\textbf{Convergence of }$(U_{n})$]\label{L_UnCtyAtT}Under the assumptions of
Theorem \ref{T_InducingWIP},
\begin{align*}
\left(  U_{n}(s)\right)  _{s\geq0}\overset{\mathcal{\mu}_{Y}}{\Longrightarrow
}\left(  W(s\,\mu(Y))\right)  _{s\geq0}\text{\quad in }(\mathcal{D}%
[0,\infty),\mathcal{M}_{1}).
\end{align*}
\end{lemma}

\begin{proof}
For $n\geq1$ and $s\in \lbrack0,\infty)$, we
let $u_{n}(s):=n^{-1}N_{\left\lfloor sn\right\rfloor }
$. Since $\left\lfloor u_{n}(s)\,n\right\rfloor =N_{\left\lfloor
sn\right\rfloor }$, we have
\begin{equation}
U_{n}(s)=P_{n}(u_{n}(s))\quad\text{on }Y\text{ for }n\geq1\text{ and }%
s\geq0\text{.}\label{Eq_TimeChange}%
\end{equation}
We regard $U_{n}$, $P_{n}$, $W$, and $u_{n}$ as random elements of
$(\mathcal{D},\mathcal{M}_{1})=(\mathcal{D}[0,\infty),\mathcal{M}_{1})$. Note
that $u_{n}\in\mathcal{D}_{\uparrow}:=\{g\in\mathcal{D}:$ $g(0)\geq0$ and
$g$\ \ non-decreasing$\}$. Let $u$ denote the constant random element of
$\mathcal{D}$ given by $u(s)(y):=s\mu(Y)$, $s\geq0$.

Recalling Lemma \ref{L_LapNumberEtc} b), we see that for $\mu_{Y}$-a.e. $y\in
Y$ we have $u_{n}(.)(y)\rightarrow u(.)(y)$ uniformly on compact subsets of
$[0,\infty)$. Hence, $u_{n}\rightarrow u$ in $(\mathcal{D},\mathcal{M}_{1})$
holds $\mu_{Y}$-a.e. In particular,
\[
u_{n}\overset{\mu_{Y}}{\Longrightarrow}u\text{\quad in }(\mathcal{D}%
,\mathcal{M}_{1})\text{.}%
\]
By assumption (\ref{Eq_WIPinduced}) we also have $P_{n}\overset{\mathcal{L}%
(\mu_{Y})}{\Longrightarrow}W$ in $(\mathcal{D},\mathcal{M}_{1})$. But then we
automatically get
\begin{equation}
(P_{n},u_{n})\overset{\mu_{Y}}{\Longrightarrow}(W,u)\text{\quad in
}(\mathcal{D},\mathcal{M}_{1})^{2}\text{,}\label{Eq_JointJ1CgePnun}%
\end{equation}
since the limit $u$ of the second component is deterministic.

The composition map $(\mathcal{D},\mathcal{M}_{1})\times(\mathcal{D}%
_{\uparrow},\mathcal{M}_{1})\rightarrow(\mathcal{D},\mathcal{M}_{1})$,
$(g,v)\mapsto g\circ v$, is continuous at every pair $(g,v)$ with
$v\in\mathcal{C}_{\uparrow\uparrow}:=\{g\in\mathcal{D}:$ $g(0)\geq0$ and $g$
strictly increasing and continuous$\}$, cf.\ \cite[Theorem 13.2.3]{Whitt}. As the
limit $(W,u)$ in (\ref{Eq_JointJ1CgePnun}) satisfies $\Pr((W,u)\in
\mathcal{D}\times\mathcal{C}_{\uparrow\uparrow})=1$, the standard mapping
theorem for distributional convergence (cf.\ \cite[Theorem 3.4.3]{Whitt}) applies
to $(P_{n},u_{n})$, showing that
\begin{align*}
P_{n}\circ u_{n}\overset{\mu_{Y}}{\Longrightarrow}W\circ u\text{\quad in
}(\mathcal{D},\mathcal{M}_{1})\text{.}
\end{align*}
In view of (\ref{Eq_TimeChange}), this is what was to be proved.%
\end{proof}%

\subsection{Control of excursions}

Passing from convergence of $(U_n)$ to convergence of $(W_n)$
requires a little preparation.

\begin{lemma} \label{lem-M1b}
(i) Let $g,g' \in \mathcal{D}[0,T] $ and
$0=T_0<\ldots<T_m=T$. Then
$$d_{\mathcal{M}_1,T}(g,g') \leq
\max_{1\leq j \leq m}  d_{\mathcal{M}_1,[T_{j-1},T_j]}
(g \vert _{[T_{j-1},T_j]},g'  \vert _{[T_{j-1},T_j]}).$$
(ii) Let $g_j \in \mathcal{D}[T_{j-1},T_j]$ and
$\bar{g}_j:=1_{[T_{j-1},T_j)}\,g_j(T_{j-1})+1_{\{T_j\}}\,g_j(T_j)$. Then
$$ d_{\mathcal{M}_1,[T_{j-1},T_j]} (g_j,\bar{g}_j) \leq
2 g_j^{\ast} + ( T_{j} -T_{j-1} ) ,$$
where
$ g_j^{\ast} :=
(\sup_{T_{j-1} \leq s \leq t \leq T_{j}} (g_j(s)-g_j(t))) \wedge
(\sup_{T_{j-1} \leq s \leq t \leq T_{j}} (g_j(t)-g_j(s))).$

\end{lemma}

\begin{proof}
The first assertion is obvious. To validate the second, assume
without loss that $j=1$ and that $g_1$ is predominantly increasing in that
$g_1^{\ast}= \sup_{T_0 \leq s \leq t \leq T_1} (g_1(s)-g_1(t))$.
In this case, $g_1^{\ast}= \sup_{T_0 \leq t \leq T_1}
(g_1^{\uparrow}(t)-g_1(t))$ for the nondecreasing function
$g_1^{\uparrow}(t):= \sup_{T_0 \leq s \leq t} g_1(s)$.
Therefore,
$$ d_{\mathcal{M}_1,[T_0,T_1]} (g_1,g_1^{\uparrow}) \leq
   \| g_1 - g_1^{\uparrow}  \|  =  g_1^{\ast}. $$
Letting $\bar{g_1}^{\uparrow}:= 1_{[T_0,T_1)}\,g_1^{\uparrow}(T_0)
+1_{\{T_1\}}\,g_1^{\uparrow}(T_1)$, it is clear that
$$  d_{\mathcal{M}_1,[T_0,T_1]} (\bar{g}_1^{\uparrow},\bar{g}_1) \leq
    \| \bar{g}_1^{\uparrow} - \bar{g}_1  \|  =
    \vert  g_1^{\uparrow}(T_1) - g_1(T_1) \vert
    \leq g_1^{\ast}.  $$
Finally, we check that
$$ d_{\mathcal{M}_1,[T_0,T_1]} (g_1^{\uparrow},\bar{g}_1^{\uparrow})
   \leq   T_1 -T_0. $$

To this end, we refer to Figure~\ref{fig-M1} where
$\Gamma_1=\Gamma(g_1^{\uparrow})$ and $\Gamma_2=\Gamma(\bar{g}_1^{\uparrow})$
represent the completed graphs of $g_1^{\uparrow}$ and $\bar{g}_1^{\uparrow}$ respectively.
Here $\Gamma_2$ consists of one horizontal line segment
followed by one vertical segment.
The picture of $\Gamma_1$ is schematic, it may also
contain horizontal and vertical line segments.

Choose $C$ on the graph of $\Gamma_1$ that is equidistant from
$AD$ and $DB$ and let $E$ be the point on $DB$ that is the same height as $C$.
Choose parametrizations $G_i=(\lambda_i,\gamma_i)$ of $\Gamma_i$, $i=1,2$, satisfying
\begin{itemize}
\item[(i)]  $G_1(0)=G_2(0)=A,\quad
G_1(1)=G_2(1)=B$,
\item[(ii)] $G_1(\frac12)=C, \quad
G_2(\frac12)=E$,
\item[(iii)] $\gamma_1(t)=\gamma_2(t)$ for all $t\in[\frac12,1]$.
\end{itemize}
Automatically $\|\lambda_1-\lambda_2\|\le |AD| =T_1 -T_0$ and by construction
$\|\gamma_1-\gamma_2\|\le |DE|\le |AD|$, as required.
\end{proof}\\

As a consequence, we obtain:

\begin{lemma} \label{lem-M1}
$d_{\mathcal{M}_1,T}(W_n,U_n)\le \max_{0\le j\le \lfloor Tn \rfloor+1} (n^{-1}r + 2 B(n)^{-1}\Phi^\ast )\circ F^j$.
\end{lemma}

\begin{proof}
Let $y\in Y$ and decompose $[0,T]$ according to the consecutive excursions,
letting $T_j :=t_{n,j}(y) \wedge T $, $j \leq m:=\lfloor Tn \rfloor+1$.
Consider $g(t):=W_n(t)(y)$, $t \in [0,T]$.
If we set $g_j:=g \vert _{[T_{j-1},T_j]} $, then
$\bar{g}_j$ as defined in Lemma \ref{lem-M1b} coincides
with $U_n(.)(y) \vert _{[T_{j-1},T_j]} $, so that
$$ d_{\mathcal{M}_1,T}(W_n(.)(y),U_n(.)(y)) \le
   \max_{0\le j\le m}
   d_{\mathcal{M}_1,[T_{j-1},T_j]}(g_j,\bar{g}_j). $$
But $ T_j - T_{j-1} =  n^{-1}r\circ F^j$, and since
$g_j(s)-g_j(t) =
B(n)^{-1} (\phi_{\ell^{\prime}}-\phi_{\ell}) \circ F^j (y)$,
for suitable $0 \leq \ell^{\prime} \leq \ell \leq r$,
we see that Lemma \ref{lem-M1b} gives
$$  d_{\mathcal{M}_1,[T_{j-1},T_j]}(g_j,\bar{g}_j) \leq
    (n^{-1}r + 2 B(n)^{-1}\Phi^\ast )\circ F^j, $$
as required.
\end{proof}

\begin{figure}
\centering
\includegraphics[height = 6cm]{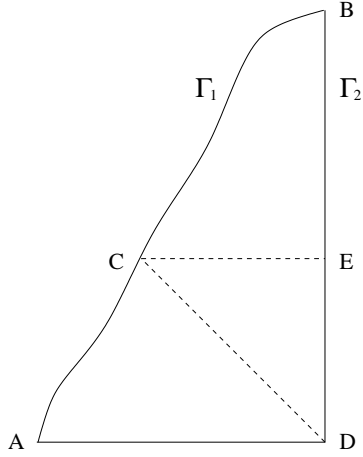}
\caption{A monotone excursion}
\label{fig-M1}
\end{figure}

\vspace{0.3cm}%

\noindent
\textbf{Proof of Theorem \ref{T_InducingWIP}.}
Fix any $T>0$.
By Lemma~\ref{L_ReturnTimeControl},
$n^{-1}\max\nolimits_{0\leq j\leq\left\lfloor Tn\right\rfloor +1}(r\circ
F^{j})\to 0$ a.e.\ and
by assumption~\eqref{Eq_EssMonotonicity}
$B(n)^{-1}\max\nolimits_{0\leq j\leq\left\lfloor Tn\right\rfloor +1}(\Phi^\ast\circ
F^{j})\overset{\mu_{Y}}{\Longrightarrow} 0$.
Hence Lemma~\ref{lem-M1} guarantees that
\begin{align} \label{eq-M1}
d_{\mathcal{M}_1,T}(W_n,U_n)\overset{\mu_{Y}}{\Longrightarrow}0.
\end{align}
Recall also from Lemma~\ref{L_UnCtyAtT} that
\begin{equation}
\left(  U_{n}(s)\right)  _{0\leq s\leq T}\overset{\mu_{Y}}{\Longrightarrow
}\left(  W(s\,\mu(Y))\right)  _{0\leq s\leq T}\text{\quad in }(\mathcal{D}%
[0,T],\mathcal{M}_{1})\text{.}\label{Eq_M1CgeWnFiniteT}%
\end{equation}
It follows (see ~\cite[Theorem~3.1]{Billingsley99})
from~\eqref{eq-M1} and~\eqref{Eq_M1CgeWnFiniteT} that
\begin{align*}
\left(  W_{n}(s)\right)  _{0\leq s\leq T}\overset{\mu_{Y}}{\Longrightarrow
}\left(  W(s\,\mu(Y))\right)  _{0\leq s\leq T}\text{\quad in }(\mathcal{D}%
[0,T],\mathcal{M}_{1})\text{.}
\end{align*}
This immediately gives $\left(  W_{n}(s)\right)  _{t\geq0}\overset{\mu_{Y}%
}{\Longrightarrow}\left(  W(s\,\mu(Y))\right)  _{s\geq0}$\ in $(\mathcal{D}%
[0,\infty),\mathcal{M}_{1})$. Strong distributional convergence as asserted in
(\ref{Eq_InducedWIPinM1}) follows via Proposition~\ref{P_AutomaticStrongDistrCge}.
$\blacksquare$%

\vspace{0.3cm}%

\section{Examples}
\label{sec-ex}

We continue to suppose that $f$ is an ergodic m.p.t.\ on a probability space
$(X,\mu)$ with first return map $F=f^r:Y\to Y$ where $\mu(Y)>0$.
Suppose further that the induced map $F:Y\to Y$ is Gibbs-Markov with ergodic
invariant probability measure $\mu_Y$ and
partition $\beta$, and that $r|_a$ is constant for each $a\in\beta$.
Let  $\phi:X\to\R$ be an $L^\infty$ mean zero observable, with
induced observable $\Phi:Y\to\R$.

\begin{theorem}\label{thm-const}
Suppose that $\phi$ is constant on $f^\ell a$ for every $a\in\beta$ and $\ell\in\{0,\dots,r|_a-1\}$.
If $\Phi$ lies in the domain of an $\alpha$-stable law, then
$\Phi$ satisfies the WIP in $(D,\mathcal{J}_1)$ with $B(n)=n^{1/\alpha}$.
If in addition condition~\eqref{Eq_EssMonotonicity} holds,  then $\phi$ satisfies
the WIP in $(D,\mathcal{M}_1)$.
\end{theorem}

\begin{proof}
By~\cite{AaronsonDenker01}, $n^{-1/\alpha}\Phi_n$ converges in distribution
to the given stable law.
The assumptions guarantee that the induced observable $\Phi$ is constant
on each $Y_j$.
Hence we can apply Tyran-Kami\'nska~\cite[Corollary~4.1]{Tyran-Kaminska10}
to deduce that $\Phi$ satisfies
the corresponding $\alpha$-stable WIP in $(D,\mathcal{J}_1)$.
In particular, condition~\eqref{Eq_WIPinduced} is satisfied.
The final statement follows from Theorem~\ref{T_InducingWIP}.
\end{proof}

\vspace{1ex}
For certain examples, including Pomeau-Manneville intermittency maps, we can work with general H\"older observables,
thus improving upon~\cite[Example~4.1]{Tyran-Kaminska10}.
The idea is to decompose the observable $\phi$ into a piecewise constant observable $\phi_0$ and a H\"older observable $\tilde\phi$ in such a way that
only $\phi_0$ ``sees'' the source of the anomalous behaviour.

In the remainder of this section, we carry out this procedure for the maps~\eqref{eq-LSV} and thereby prove Theorem~\ref{thm-LSV}.
(Lemma~\ref{lem-LY} and Proposition~\ref{prop-GMCLT} below hold in the general context of induced Gibbs-Markov maps.)

Fix $\theta\in(0,1)$ and let $d_\theta$ denote the
symbolic metric on $Y$, so $d_\theta(x,y)=\theta^{s(x,y)}$ where $s(x,y)$ is
the least integer $n\ge0$ such that $F^nx,F^ny$ lie in distinct elements
of $\beta$.
An observable $\Phi:Y\to\R$ is {\em piecewise Lipschitz} if 
$D_a(\Phi):=\sup_{x,y \in a, x\neq y}|\Phi(x)-\Phi(y)|/d_\theta(x,y) 
< \infty$ for each $a\in\beta$, and {\em Lipschitz} if  
$\|\Phi\|_\theta=|\Phi|_\infty+ \sup_{a\in\beta}  D_a(\Phi)<\infty$.  The space
$\Lip$ of Lipschitz observables $\Phi:Y\to\R$ is a Banach space.
Note that $\Phi$ is integrable with  
$\sum_{a\in\beta}\mu_Y(a)D_a(\Phi) < \infty$ if and only if  
$\sum_{a\in\beta}\mu_Y(a)\|1_a\Phi\|_\theta<\infty$.

Let $L$ denote the transfer operator for $F:Y\to Y$.

\begin{lemma} \label{lem-LY}
\textbf{a)}
The essential spectral radius of $L:\Lip\to\Lip$ is at most $\theta$.
\newline\textbf{b)}
Suppose that $\Phi:Y\to\R$ is a piecewise Lipschitz observable
satisfying
$\sum_{a\in\beta}\mu_Y(a)\|1_a\Phi\|_\theta<\infty$.
Then $L\Phi\in\Lip$.
\end{lemma}

\begin{proof}
This is standard.   See for example~\cite[Theorem~1.6]{AaronsonDenker01} for part (a) and~\cite[Lemma~2.2]{MN05} for part~(b).
\end{proof}

\begin{proposition} \label{prop-GMCLT}
Let $\Phi:Y\to\R$ be a piecewise Lipschitz mean zero observable lying in $L^p$,
for some $p\in(1,2)$.  Assume that
$\sum_{a\in\beta}\mu_Y(a)\|1_a\Phi\|_\theta<\infty$.
Then $\max_{j=0,\dots,n-1}n^{-\gamma}\Phi_j\to_d0$ for all $\gamma>1/p$.
\end{proposition}

\begin{proof}
Suppose first that $F$ is weak mixing (this assumption is removed below).
Then $L:\Lip\to\Lip$ has no eigenvalues on the unit circle
except for the simple eigenvalue at $1$ (corresponding to constant functions).
By Lemma~\ref{lem-LY}(a), there exists $\tau<1$ such that the remainder of the
spectrum of $L$ lies strictly inside the ball of radius $\tau$.
In particular, there is a constant $C>0$ such that
$\|L^nv-\int v\,d\mu_Y\|\le C\tau^n\|v\|$ for all $v\in\Lip$, $n\ge1$.

By Lemma~\ref{lem-LY}(b), $L\Phi\in\Lip$.
Hence $\chi=\sum_{j=1}^\infty L^j\Phi\in\Lip$.
Following Gordin~\cite{Gordin69}, write $\Phi=\hat \Phi+\chi\circ F-\chi$.   
Then $\hat \Phi\in L^p$ (since $\chi\in\Lip$
and $\Phi\in L^p$).
Applying  $L$ to both sides and noting that $L(\chi\circ F)=\chi$,
we obtain that $L\hat \Phi=0$.
It follows that the sequence $\{\hat \Phi_n;\,n\ge1\}$ defines
a reverse martingale sequence.

By Burkholder's
inequality~\cite[Theorem~3.2]{Burkholder73},
\begin{align*}
|\hat\Phi_n|_p & \ll \Bigl\|\Bigl(\sum_{j=1}^n\hat\Phi^2\circ F^j\Bigr)^{1/2}\Bigr\|_p
=\Bigl(\int\Bigl(\sum_{j=1}^n\hat\Phi^2\circ F^j\Bigr)^{p/2}\Bigr)^{1/p}
\\ &
\le \Bigl(\int\sum_{j=1}^n|\hat\Phi|^p\circ F^j\Bigr)^{1/p}=|\hat\Phi|_p n^{1/p}.
\end{align*}
By Doob's inequality~\cite{Doob53} (see also~\cite[Equation~(1.4), p.~20]{Burkholder73},
$|\max_{j=0,\dots,n-1}\hat\Phi_j|_p\ll |\hat\Phi_n|_p\ll n^{1/p}$.
By Markov's inequality, for $\epsilon>0$ fixed,
$\mu_Y(|\max_{j=0,\dots,n-1}\hat\Phi_j|\ge \epsilon n^\gamma)\le |\max_{j=0,\dots,n-1}\hat\Phi_j|_p^p/(\epsilon^pn^{\gamma p})
\ll n^{-(\gamma p-1)}\to0$ as $n\to\infty$.
Hence $\max_{j=0,\dots,n-1}n^{-\gamma}\hat\Phi_j\to_d0$.
Since $\Phi$ and $\hat\Phi$ differ by a bounded coboundary, 
$\max_{j=0,\dots,n-1}n^{-\gamma}\Phi_j\to_d0$ as required.

It remains to remove the assumption about eigenvalues (other than $1$)
for $L$ on the unit circle.    Suppose that there are $k$ such eigenvalues $e^{i\omega_\ell}$,
$\omega_\ell\in(0,2\pi)$, $\ell=1,\dots,k$ (including multiplicities).
Then we can write $\Phi=\Psi_0+\sum_{\ell=1}^k\Psi_\ell$ where
$\|L^n\Psi_0\|_\theta\le C\tau^n\|\Psi_0\|_\theta$ and $L \Psi_\ell=e^{i\omega_\ell}\Psi_\ell$.
In particular, the above argument applies to $\Psi_0$, while
$L\Psi_\ell=e^{i\omega_\ell}\Psi_\ell$, $\ell=1,\dots k$.

A simple argument (see~\cite{MN04b}) shows that $\Psi_\ell \circ F
= e^{-i\omega_\ell}\Psi_\ell$ for $\ell=1,\dots,k$, so that
$|\sum_{j=1}^n \Psi_\ell\circ F^j|_\infty\le 2|e^{i\omega_\ell}-1|^{-1}|\Psi_\ell|_\infty$
which is bounded in $n$.   Hence the estimate for $\Phi$ follows from the one for
$\Psi_0$.
\end{proof}

\vspace{1ex}
\begin{proof}[\textbf{Proof of Theorem~\ref{thm-LSV}}]
We verify the hypotheses of Theorem~\ref{T_InducingWIP}.
A convenient inducing set for the maps~\eqref{eq-LSV} is $Y=[\frac12,1]$.
Let $\phi_0=\phi(0)-\mu(Y)^{-1}\phi(0)1_Y$.
(The first term is the important one, and the second term is simply an arbitrary
choice that ensures that $\phi_0$ has mean zero while preserving the piecewise
constant requirement in Theorem~\ref{thm-const}.)
Write $\phi=\phi_0+\tilde \phi$ and note that
$\tilde\phi$ is a mean zero piecewise H\"older observable vanishing at $0$.
We have the corresponding decomposition $\Phi=\Phi_0+\tilde\Phi$ for
the induced observables.
By Theorem~\ref{thm-const},
$\Phi_0$ satisfies the WIP (in the $\mathcal{J}_1$ topology).

Let $\eta$ denote the H\"older exponent of $\phi$.
By the proof of~\cite[Theorem~1.3]{Gouezel04}, $\tilde \phi$ induces to a
piecewise Lipschitz mean zero observable
$\tilde \Phi$ satisfying $\sum_{a\in\beta}\mu_Y(a)\|1_a\tilde\Phi\|_\theta<\infty$
for suitably chosen $\theta$.
Moreover~\cite{Gouezel04} shows that $\tilde\Phi$ lies in $L^2$ provided that
$\eta>\gamma-\frac12$.
Exactly the same argument shows that
$\tilde \Phi$ lies in $L^p$ provided $\eta>\gamma-\frac{1}{p}$.
In particular, for any $\eta>0$, there exists $p>1/\gamma$ such that
$\tilde\Phi\in L^p$.
Since we are normalising by $B(n)=n^{1/\alpha}=n^{\gamma}$, it follows
from Proposition~\ref{prop-GMCLT} that $\tilde \Phi$ does not
contribute to the WIP.

Combining the results for $\Phi_0$ and $\tilde\Phi$, we deduce that
$\Phi$ satisfies the WIP (in the $\mathcal{J}_1$ topology).
In particular, condition~\eqref{Eq_WIPinduced} is satisfied.

It remains to verify condition~\eqref{Eq_PointwiseMonotonicityCtrl}.
In fact, we show that $\Phi^\ast$ is bounded.
Suppose that $\phi(0)>0$ (the case $\phi(0)<0$ is treated similarly).
Choose $\epsilon>0$ such that $\phi>0$ on $[0,\epsilon]$.
Define the decreasing sequence $x_n\in(0,\frac12)$ where $f(x_n)=x_{n-1}$,
$x_1=\frac12$,
and let $k$ be such that $x_n\in(0,\epsilon)$ for all $n\ge k$.
Then for $y\in Y$, $f^\ell y\in[0,\epsilon]$ for $0\le \ell\le r(y)-k$.

Now observe that
\begin{itemize}
\item[(i)]
$\phi_{\ell^\prime}(y)-\phi_\ell(y)\le 0$ for $1\le\ell^\prime\le\ell\le r(y)-k$,
\item[(ii)] $\phi_{\ell^\prime}(y)-\phi_\ell(y)\le k|\phi|_\infty$ for $r(y)-k\le\ell^\prime\le\ell\le r(y)$,
\item[(iii)] $\phi_{\ell^\prime}(y)-\phi_\ell(y)\le
\phi_{r(y)-k}(y)-\phi_\ell(y)\le
k|\phi|_\infty$ for $1\le\ell^\prime\le r(y)-k\le\ell\le r(y)$,
\end{itemize}
Hence
\begin{align*}
\Phi^{\ast}(y) & \le
\max_{0\le\ell^\prime\le\ell\le r(y)}(\phi_{\ell^\prime}(y)-\phi_\ell(y))
\\ & \le  |\phi|_\infty +\max_{1\le\ell^\prime\le\ell\le r(y)}(\phi_{\ell^\prime}(y)-\phi_\ell(y))
\le (k+1)|\phi|_\infty,
\end{align*}
as required.
\end{proof}

\begin{remark}
The arguments in the proof of Theorem~\ref{thm-LSV} apply to a much wider class
of examples, including intermittent maps with neutral periodic points
or with multiple neutral fixed/periodic points.
In such cases, condition~\eqref{Eq_EssMonotonicity} is again automatically satisfied.
\end{remark}

\paragraph{Acknowledgements}
The research of IM was supported in part by EPSRC Grant EP/F031807/1 held at the University of Surrey.
We are very grateful to the referee for very helpful suggestions that led to a significantly simplified proof of the main result in this paper.


\begin{thebibliography}{10}

\bibitem{AaronsonDenker01}
J.~Aaronson and M.~Denker. {Local limit theorems for partial sums of stationary
  sequences generated by Gibbs-Markov maps}. \emph{Stoch. Dyn.} \textbf{1}
  (2001) 193--237.

\bibitem{AvramTaqqu92}
F.~Avram and M.~S. Taqqu. Weak convergence of sums of moving averages in the
  {$\alpha$}-stable domain of attraction. \emph{Ann. Probab.} \textbf{20}
  (1992) 483--503.

\bibitem{BalintChernovDolgopyat11}
P.~B{\'a}lint, N.~Chernov and D.~Dolgopyat. Limit theorems for dispersing
  billiards with cusps.
  \emph{Comm. Math. Phys.} \textbf{308} (2011) 479--510.

\bibitem{BalintGouezel06}
P.~B{\'a}lint and S.~Gou{\"e}zel. Limit theorems in the stadium billiard.
  \emph{Comm. Math. Phys.} \textbf{263} (2006) 461--512.

\bibitem{BM08}
P.~B{\'a}lint and I.~Melbourne. Decay of correlations and invariance principles
  for dispersing billiards with cusps, and related planar billiard flows.
  \emph{J. Stat. Phys.} \textbf{133} (2008) 435--447.

\bibitem{BenArousCerny07}
G.~Ben~Arous and J.~{\v{C}}ern{\'y}. Scaling limit for trap models on {$\mathbb
  Z^d$}. \emph{Ann. Probab.} \textbf{35} (2007) 2356--2384.

\bibitem{BenedicksYoung00}
M.~Benedicks and L.-S. Young. Markov extensions and decay of correlations for
  certain {H}\'enon maps. \emph{Ast\'erisque} (2000) no.~261, 13--56.

\bibitem{Billingsley99}
P.~Billingsley. \emph{Convergence of probability measures}, second ed., Wiley
  Series in Probability and Statistics: Probability and Statistics, John Wiley
  \& Sons Inc., New York, 1999.

\bibitem{Burkholder73}
D.~L. Burkholder. Distribution function inequalities for martingales.
  \emph{Ann. Probability} \textbf{1} (1973) 19--42.

\bibitem{ConzeBorgne01}
J.-P. Conze and S.~Le Borgne. {M{\'e}thode de martingales et flow
  g{\'e}od{\'e}sique sur une surface de courbure constante n{\'e}gative}.
  \emph{Ergodic Theory Dynam. Systems} \textbf{21} (2001) 421--441.

\bibitem{DedeckerMerlevede}
J.~Dedecker and F.~Merlev\`{e}de. {Weak invariance principle and exponential
  bounds for some special functions of intermittent maps}. \emph{High
  dimensional probability} \textbf{5} (2009) 60--72.

\bibitem{DenkerPhilipp84}
M.~Denker and W.~Philipp. {Approximation by Brownian motion for Gibbs measures
  and flows under a function}. \emph{Ergodic Theory Dynam. Systems} \textbf{4}
  (1984) 541--552.

\bibitem{Doob53}
J. L. Doob.  {\em Stochastic processes}.   Wiley, New York, 1953.

\bibitem{Eagleson76}
G.~K. Eagleson. {Some simple conditions for limit theorems to be mixing}.
  \emph{Teor. Verojatnost. i Primenen} \textbf{21} (1976) 653--660.

\bibitem{FMT03}
M.~J. Field, I.~Melbourne and A.~T{\" o}r{\" o}k. {Decay of correlations,
  central limit theorems and approximation by Brownian motion for compact Lie
  group extensions}. \emph{Ergodic Theory Dynam. Systems} \textbf{23} (2003)
  87--110.

\bibitem{Gordin69}
M.~I. Gordin. {The central limit theorem for stationary processes}.
  \emph{Soviet Math. Dokl.} \textbf{10} (1969) 1174--1176.

\bibitem{GMsub}
G.~A. Gottwald and I.~Melbourne. {Central limit theorems and suppression of
  anomalous diffusion for systems with symmetry}. Preprint.

\bibitem{Gouezel04}
S.~Gou{\"e}zel. {Central limit theorem and stable laws for intermittent maps}.
  \emph{Probab. Theory Relat. Fields} \textbf{128} (2004) 82--122.

\bibitem{Gouezel07}
S.~Gou{\"e}zel. {Statistical properties of a skew product with a curve of
  neutral points}. \emph{Ergodic Theory Dynam. Systems} \textbf{27} (2007)
  123--151.

\bibitem{Gouezel10}
S.~Gou{\"e}zel. Almost sure invariance principle for dynamical systems by
  spectral methods. \emph{Ann. Probab.} \textbf{38} (2010) 1639--1671.

\bibitem{HofbauerKeller82}
F.~Hofbauer and G.~Keller. {Ergodic properties of invariant measures for
  piecewise monotonic transformations}. \emph{Math. Z.} \textbf{180} (1982)
  119--140.

\bibitem{HM07}
M.~Holland and I.~Melbourne. {Central limit theorems and invariance principles
  for Lorenz attractors}. \emph{J. London Math. Soc.} \textbf{76} (2007)
  345--364.

\bibitem{LiveraniSaussolVaienti99}
C.~Liverani, B.~Saussol and S.~Vaienti. {A probabilistic approach to
  intermittency}. \emph{Ergodic Theory Dynam. Systems} \textbf{19} (1999)
  671--685.

\bibitem{MN04b}
I.~Melbourne and M.~Nicol. Statistical properties of endomorphisms and compact
  group extensions. \emph{J. London Math. Soc.} \textbf{70} (2004) 427--446.

\bibitem{MN05}
I.~Melbourne and M.~Nicol. Almost sure invariance principle for nonuniformly
  hyperbolic systems. \emph{Commun. Math. Phys.} \textbf{260} (2005) 131--146.

\bibitem{MN09}
I.~Melbourne and M.~Nicol. A vector-valued almost sure invariance principle for
  hyperbolic dynamical systems. \emph{Ann. Probability} \textbf{37} (2009)
  478--505.

\bibitem{MT02}
I.~Melbourne and A.~T{\" o}r{\" o}k. {Central limit theorems and invariance
  principles for time-one maps of hyperbolic flows}. \emph{Commun. Math. Phys.}
  \textbf{229} (2002) 57--71.

\bibitem{MT04}
I.~Melbourne and A.~T{\" o}r{\" o}k. {Statistical limit theorems for suspension
  flows}. \emph{Israel J. Math.} \textbf{144} (2004) 191--209.

\bibitem{PomeauManneville80}
Y.~Pomeau and P.~Manneville. Intermittent transition to turbulence in
  dissipative dynamical systems. \emph{Comm. Math. Phys.} \textbf{74} (1980)
  189--197.

\bibitem{Ratner73}
M.~Ratner. {The central limit theorem for geodesic flows on $n$-dimensional
  manifolds of negative curvature}. \emph{Israel J. Math.} \textbf{16} (1973)
  181--197.


\bibitem{Skorohod56}
A.~V. Skorohod. Limit theorems for stochastic processes. \emph{Teor.
  Veroyatnost. i Primenen.} \textbf{1} (1956) 289--319.

\bibitem{SzaszVarju07}
D.~Sz{\'a}sz and T.~Varj{\'u}. Limit laws and recurrence for the planar
  {L}orentz process with infinite horizon. \emph{J. Stat. Phys.} \textbf{129}
  (2007) 59--80.

\bibitem{Tyran-Kaminska10}
M.~Tyran-Kami{\'n}ska. Weak convergence to {L}\'evy stable processes in
  dynamical systems. \emph{Stoch. Dyn.} \textbf{10} (2010) 263--289.

\bibitem{Whitt}
W.~Whitt. \emph{Stochastic-process limits}. Springer Series in Operations
  Research, Springer-Verlag, New York, 2002.

\bibitem{Young98}
L.-S. Young. Statistical properties of dynamical systems with some
  hyperbolicity. \emph{Ann. of Math.} \textbf{147} (1998) 585--650.

\bibitem{Young99}
L.-S. Young. Recurrence times and rates of mixing. \emph{Israel J. Math.}
  \textbf{110} (1999) 153--188.

\bibitem{Zweimuller03}
R.~Zweim{\"u}ller. Stable limits for probability preserving maps with
  indifferent fixed points. \emph{Stoch. Dyn.} \textbf{3} (2003) 83--99.

 \bibitem{Zweimueller07}
 R.~Zweim{\"u}ller. Mixing limit theorems for ergodic transformations. \emph{J.
   Theoret. Probab.} \textbf{20} (2007) 1059--1071.


\end{thebibliography}
\end{document}